\documentclass[a4paper,11pt,oneside,openany,reqno]{amsart}

\usepackage{latexsym,amsmath,amssymb,amsfonts}
\usepackage{bm}
\usepackage[dvipdfmx]{graphicx}
\usepackage{subfigure}
\usepackage{verbatim}
\usepackage{wrapfig}
\usepackage{ascmac}
\usepackage{cases}
\usepackage{amsthm}
\usepackage{color}

\usepackage{comment}
\usepackage{ulem}
\usepackage {cancel}

  \makeatletter
    
    \@addtoreset{equation}{section}
  \makeatother
\newtheorem{theo}{Theorem}[section]

\newtheorem{lemm}[theo]{Lemma}

\newcounter{c}
\newcounter{d}
\newcounter{b}

\newcommand{\Na}{\mathbb{N}} 
\newcommand{\R}{\mathbb{R}} 

\newcommand{\e}{\varepsilon} 
\newcommand{\supp}{\mathrm{supp}} 
\newcommand{\dist}{\mathrm{dist}} 

\title[fast reaction limit]{Interface disappearance in fast reaction limit}
\author{Yuki Tsukamoto}
\address{Meiji University,
	164-8525, Tokyo, Japan}
\email{math.y.tsukamoto@gmail.com}
\date{}


\begin{document}
\maketitle
\begin{abstract}	
We study the singular limit problem referred to as the fast reaction limit.
This problem has been extensively studied when the same reaction term is used in a two-component system. However, the behavior of the solution under different reaction terms remains not yet well understood. In this paper, we will consider the problem where the reaction term is represented by a power term. We prove that the initial interface disappears immediately, and the function converges to a solution that satisfies the heat equation.
\end{abstract}

	\section{Introduction}
Reaction-diffusion systems represent a class of phenomena where the dynamics of time and space intricately intertwine. Their expressiveness and rich structural characteristics find applications spanning from chemical reactions to the evolution of ecosystems \cite{B03, E80, E10}. These systems model processes where molecules diffuse and chemical reactions unfold, shedding light on patterns and phenomena observed in our environment.
Many researchers studied the reaction-diffusion system with the same reaction term
\begin{align} \label{intoro1}
	\begin{cases}
			\partial_t u_k = \Delta u_k - k F(u_k,v_k) & \mathrm{in} \ Q_T := \Omega \times (0,T],\\
			\partial_t v_k = -k F(u_k,v_k) & \mathrm{in} \ Q_T,
	\end{cases}
\end{align}
where $\Omega$ is a bounded domain in $\R^n$ with smooth boundary $\partial \Omega$, $F$ is non-decreasing function, and $T$ is a positive constant.
The term $-k F(u_k,v_k)$ is referred to as the reaction term.
 A typical example of $F$ is
$u^{m_1}v^{m_2}$, with constants $m_1,m_2 \geq 1$.
Hilhorst et al. \cite{H96, H97} proved that the limit solutions $u_\infty$ and $v_\infty$ of \eqref{intoro1} become solutions to the one-phase Stefan problem under the one-dimension case.
Thus, the solutions $u_\infty$ and $v_\infty$ separate and become solutions to the free boundary problem.
Eymard et al. \cite{E01} obtained similar results under the Dirichlet boundary condition in a general dimension.
Evans \cite{E80} studied \eqref{intoro1} with the addition of $\Delta v_k$ in the equation for $v_k$ and setting $F(u_k,v_k)=u_k v_k$.
The equation \eqref{intoro1} can be transformed by
\begin{align*}
	\partial_t u_k- \Delta u_k= \partial_t v_k.
\end{align*}
We call to  a balanced fast reaction pair when such the equation holds.
This equation plays an important role in demonstrating the convergence of the limit solution.
Otherwise, we call it an unbalanced fast reaction pair.
Conti et al. \cite{C05} and Hilhorst et al. \cite{H08} studied
multi-component competition-diffusion systems that correspond to unbalanced cases.
Iida et al. \cite{I17} studied the unbalanced fast reaction problem
\begin{align}
	\begin{cases}
			\partial_t u_k= \Delta u_k -k u_k^{m_1} v_k^{m_2}
			& \mathrm{in} \ Q_T,\\
			\partial_t v_k = -k u_k^{m_3} v_k^{m_4}
			& \mathrm{in} \ Q_T,\\
			\partial_\nu u_k=0& \mathrm{on} \ S_T :=\partial \Omega \times (0,T],\\
			u_k(\cdot,0)=u_0, \ v_k(\cdot,0)=v_0
			& \mathrm{in} \ \Omega,
	\end{cases}
\end{align}
where $\nu$ is the outward unit normal vector to $\partial \Omega$, $m_i\geq 1$ $(i=1,2,3,4)$,
and $u_0$ and $v_0$ are the initial data such that $u_0,v_0 \geq0$, $u_0v_0 \equiv 0$, $u_0 \not\equiv 0$, and $v_0 \not\equiv 0$.
Since $u_0v_0 \equiv 0$, the areas occupied by $u_0$ and $v_0$ are mutually exclusive. In essence, there is an interface present at the initial time.
They obtained results for the following four cases:
\begin{align*}
	\mathrm{Case \ I} &: (m_1,m_2,m_3,m_4)=(m_1,1,1,1) \quad 3<m_1,\\
	\mathrm{Case \ II} &: (m_1,m_2,m_3,m_4)=(1,m_2,1,1) \quad 1 \leq m_2,\\
	\mathrm{Case \ III} &: (m_1,m_2,m_3,m_4)=(1,1,m_3,1) \quad 1 < m_3,\\
	\mathrm{Case \ IV} &: (m_1,m_2,m_3,m_4)=(1,1,1,m_4) \quad 1 \leq m_4<2.
\end{align*}
In Case I, $u_k$ converges to a solution of the heat equation in $Q_T$, and $v_k$ converges to zero. This means that the initial interface
vanishes instantaneously. In Case II and IV,
$u_k$ converges to a solution of the free boundary problem associated with the one-phase Stefan problem.
In Case III, $u_k$ converges to a solution of the Dirichlet problem where the boundary remains stationary from the initial values. However, they were unable to provide a proof for cases when $1<m_1\leq 3$ and $2\leq m_4$.
We proved the convergence of solutions when $2<m_1$.
\subsection{Main Theorems}
The following theorem is the main result of this study.
Throughout this study, we assume that $\Omega \subset \R^n$ is a bounded domain with
smooth boundary $\partial \Omega$. 
\begin{theo} \label{theorem1}
Suppose that $u_k$ and $v_k$ satisfy
\begin{align} \label{main}
	\begin{cases}
		\partial_t u_k=\Delta u_k - ku_k^m v_k & \  \mathrm{in} \ Q_T,\\
		\partial_t v=-ku_kv_k & \  \mathrm{in} \ Q_T,\\
		\partial_\nu u=0 & \  \mathrm{on} \ S_T,\\
		u_k(\cdot,0)=u_0, \ v_k(\cdot,0)=v_0 & \  \mathrm{on} \ \Omega.
	\end{cases}
\end{align}
Suppose that $u_0 \in C^{2}(\overline{\Omega})$, $v_0 \in C^{\alpha}(\overline{\Omega})$, $\partial_\nu u_0=0$ on $\partial \Omega$,
$u_0 \geq 0$,
$u_0 v_0 \equiv 0$, $u_0 \not\equiv 0$, and $v_0 \not\equiv 0$.
If $2<m \leq 3$, there exist a function $u_\infty \in C^{2,1}(\overline{Q_T})$ such that
$u_k$ converges uniformly to $u_\infty$ in $\overline{Q_T}$ and $u_\infty$ satisfies
\begin{align} \label{heat}	\begin{cases}
	\partial_t	u_\infty=\Delta u_\infty  & \  \mathrm{in} \ Q_T,\\
	\partial_\nu u_\infty=0 & \  \mathrm{on} \ S_T,\\
	u_\infty(\cdot,0)=u_0 & \  \mathrm{on} \ \Omega.
\end{cases}
\end{align}
Moreover, for any constant $\rho \in (0,T)$, $v_k$ converges uniformly to $0$ in $\overline{\Omega} \times [\rho,T]$.
\end{theo}
Moreover, we established the proof under Dirichlet boundary conditions.
\begin{theo} \label{theorem2}
	Suppose that $u_k$ and $v_k$ satisfy
	\begin{align} \label{mainD}
		\begin{cases}
		\partial_t u_k=\Delta u_k - ku_k^m v_k & \  \mathrm{in} \ Q_T,\\
		\partial_t v=-ku_kv_k & \  \mathrm{in} \ Q_T,\\
		 u_k=0 & \  \mathrm{on} \ S_T,\\
		u_k(\cdot,0)=u_0, \ v_k(\cdot,0)=v_0 & \  \mathrm{on} \ \Omega.
		\end{cases}
	\end{align}
	Suppose that $u_0 \in C^{2}(\overline{\Omega})$, $v_0 \in C^{\alpha}(\overline{\Omega})$, $ u_0=0$ on $\partial \Omega$,
	$u_0 \geq 0$,
	$u_0 v_0 \equiv 0$, $u_0 \not\equiv 0$, and $v_0 \not\equiv 0$.
	If $2<m $, there exist a function $u_\infty \in C^{2,1}(\overline{Q_T})$ such that
	$u_k$ converges uniformly to $u_\infty$ in $\overline{Q_T}$ and $u_\infty$ satisfies
	\begin{align} \label{heat2}	\begin{cases}
			\partial_t	u_\infty=\Delta u_\infty  & \  \mathrm{in} \ Q_T,\\
			\partial_\nu u_\infty=0 & \  \mathrm{on} \ S_T,\\
			u_\infty(\cdot,0)=u_0 & \  \mathrm{on} \ Q_T.
		\end{cases}
	\end{align}
	Moreover, for any constant $\rho \in (0,T)$ and any domain $\Omega_0 \subset \subset \Omega$, $v_k$ converges uniformly to $0$ in $\overline{\Omega_0} \times [\rho,T]$.
\end{theo}
Obtaining an estimate of $ku_k^m v_k$ is important for proving these theorems.
Iida et al. \cite{I17} obtained this estimate by proving that $u_k$ is bounded away from zero, independently of $k$, when $t$ is positive.
Under the Dirichlet boundary condition, a slight modification in the proof is necessary because $u_k$ approaches zero near the boundary.

The remainder of this paper is organized as follows. In Section 2, we first
prove the problem with the Neumann boundary condition.
We obtained an estimate of $ku_k^m v_k$ by appropriately readjusting the initial function $u_0$. Using this result, we prove Theorem \ref{theorem1}.
In Section 3, we prove the convergence of the Dirichlet boundary problem by modifying lemmas related to the Neumann boundary condition.
Since $u_k$ is zero on the boundary, we can obtain a different estimate in the boundary neighborhood from $ku_k^m v_k$, which allows us to prove Theorem \ref{theorem2}.

	\section{Neumann boundary problem}
	\begin{lemm}\label{l1l}
		If $2<m<3$, there exists a function $\underline{u_0} \in C^\infty (\Omega)$ satisfying
		\begin{align}\label{l1e0}
			\begin{cases}
								&\underline{u_0}\not\equiv 0,\\
				&0\leq \underline{u_0}\leq \min \{u_0,1\},\\
				&\supp(\underline{u_0}) \subset\subset \Omega,\\
				&\underline{u_0}^{3-m} \geq (m-1)\Delta \underline{u_0}.
			\end{cases}
		\end{align}If $m=3$, instead of $\underline{u_0}^{3-m} \geq (m-1)\Delta \underline{u_0}$, we have
	\begin{align*}
	1	\geq 2\Delta \underline{u_0} \quad \mathrm{in} \ Q_T.
\end{align*}
	\end{lemm}
\begin{proof}
	By the assumption of $u_0$, there exist a point $x_0 \in \Omega$ and a constant $0<r_0<1$ such that $B(x_0,r_0)\subset \subset \Omega$ and $u_0(x) \geq \frac{u_0(x_0)}{2}$ in $B(x_0,r_0)$. Define
	\begin{align*}
		U(x,r):= 
		\begin{cases}
			e^{\frac{-1}{r^2-|x-x_0|^2}} &\quad \mathrm{if} \ x\in B(x_0,r),\\
			0 &\quad \mathrm{if} \ x\in \Omega \backslash B(x_0,r).
		\end{cases}
	\end{align*}
	for $(x,r)\in \Omega \times (0,r_0)$. For each $0<r<r_0$,
	we have $U(\cdot,r)\in C^{\infty}(\Omega)$, $U(\cdot,r) \not\equiv 0$, $\supp(U(\cdot,r)) \subset\subset \Omega$, and $U(\cdot,r)\geq 0$. We compute
	\begin{align}\label{l1e2}
		\begin{split}
				\Delta_x U(x,r) =& \sum_{i=1}^{n} \partial_{x_i}^2 U(x,r)\\
		=& U(x,r)\{4|x-x_0|^2 (r^2-|x-x_0|^2)^{-4}\\
		&-8|x-x_0|^2 (r^2-|x-x_0|^2)^{-3}\\
		&-2n (r^2-|x-x_0|^2)^{-2} \}\\
		\leq& (12+2n)(r^2-|x-x_0|^2)^{-4} U(x,r).
		\end{split}
	\end{align}
By \eqref{l1e2}, if $2<m<3$, we obtain
\begin{align*}
	&U(x,r_1)^{3-m} -(m-1)\Delta_x U(x,r_1)\\
	\geq & U(x,r_1)^{3-m}(1-(m-1)(12+2n)(r_1^2-|x-x_0|^2)^{-4} U(x,r_1)^{m-2})\\
	\geq&0.
\end{align*}
 for sufficiently small $r_1>0$. If $m=3$, we similarly have
 \begin{align*}
 	1-2\Delta_x U(x,r_1) \geq 0
 \end{align*}
 By defining
\begin{align*}
	\underline{u_0} (x):=\min\{\frac{u_0(x_0)}{2},1\}  U(x,r_1),
\end{align*}
this lemma follows.
\end{proof}
	\begin{lemm}\label{l2l}
		There exists a function $\underline{u} \in C^{\infty}(\overline{Q_T})$ satisfying
		\begin{align} \label{heatunder}	\begin{cases}
				\partial_t \underline{u}=\Delta \underline{u}  & \  \mathrm{in} \ Q_T,\\
				\partial_\nu \underline{u}=0 & \  \mathrm{on} \ S_T,\\
				\underline{u}(\cdot,0)=\underline{u_0} & \  \mathrm{on} \ \Omega,
			\end{cases}
		\end{align}
	where $\underline{u_0}$ satisfies \eqref{l1e0}.
	Moreover, if $2<m<3$, it holds that
	\begin{align}
		&\underline{u}(x,t)>0 \quad\mathrm{in} \ \overline{\Omega} \times (0,T], \label{l2a}\\
		&\underline{u}(x,t)^{3-m} \geq (m-1)\partial_t
		\underline{u}(x,t)\quad \mathrm{in} \ Q_T. \label{l2b}
	\end{align}
If $m=3$, instead of \eqref{l2b}, we have
\begin{align*}
1	\geq 2\partial_t
	\underline{u}(x,t)\quad \mathrm{in} \ Q_T.
\end{align*}
	\end{lemm}
	\begin{proof}
		By the standard parabolic theory, there exists a classical solution of \eqref{heatunder} (see Chapter 4 in \cite{L68} ). Using the strong maximum principle and the Hopf lemma, the function \underline{u} is positive in $\overline{\Omega} \times (0,T]$. We first consider the case of $2<m<3$. Define
		$w_1(x,t):=(m-1)\Delta \underline{u}(x,t)$ and
		$w_2(x,t):= \underline{u}(x,t)^{3-m}$. Then we have
				\begin{align} \label{l2e1}	\begin{cases}
					\partial_t w_1=\Delta w_1  & \  \mathrm{in} \ Q_T,\\
					\partial_\nu w_1=0 & \  \mathrm{on} \ S_T,\\
					w_1(\cdot,0)=(m-1)\Delta \underline{u_0} & \  \mathrm{on} \ \Omega.
				\end{cases}
			\end{align}
		Moreover, we compute
		\begin{align*}
			&\partial_t w_2 = (3-m)\underline{u}^{2-m} \partial_t \underline{u},\\
			&\Delta w_2 = (3-m)\underline{u}^{2-m} \Delta \underline{u} +(3-m)(2-m)\underline{u}^{1-m}
			\nabla \underline{u} \cdot \nabla \underline{u}.
		\end{align*}
	By \eqref{heatunder} and \eqref{l2a}, we obtain
	\begin{align}\label{l2e2}	\begin{cases}
			\partial_t w_2=\Delta w_2
			+(m-2) \frac{\nabla \underline{u} \cdot \nabla \underline{u}}{\underline{u}}  & \  \mathrm{in} \ Q_T,\\
			\partial_\nu w_2=0 & \  \mathrm{on} \ S_T,\\
			w_2(\cdot,0)=\underline{u_0}^{3-m} & \  \mathrm{on} \ \Omega.
		\end{cases}
	\end{align}
Using the comparison principle (Chapter 2 in \cite{L96}), we obtain $w_2\geq w_1$ in $Q_T$ since $\underline{u_0}^{3-m} \geq (m-1)\Delta \underline{u_0}$ and $(m-2) \frac{\nabla \underline{u} \cdot \nabla \underline{u}}{\underline{u}}\geq 0$. If $m = 3$, we obtain
\begin{align*}
	1	\geq 2\partial_t
	\underline{u}(x,t)\quad \mathrm{in} \ Q_T
\end{align*}
by \eqref{l1e0}, \eqref{l2e1}, and the weak maximum principle.
 Thus this lemma follows.
	\end{proof}
By solving $\partial_t v_k=-ku_kv_k$, instead of \eqref{main}, we can consider
\begin{align} \label{main2a}
	\begin{cases}
		\partial_t u_k=\Delta u_k - kv_0u^m_k e^{-k\int_{0}^{t} u_k \ d\tau } & \  \mathrm{in} \ Q_T,\\
		\partial_\nu u_k=0 & \  \mathrm{on} \ S_T,\\
		u_k(\cdot,0)=u_0  & \  \mathrm{on} \ Q_T.
	\end{cases}
\end{align}
	\begin{lemm}\label{l3l}
		Suppose that $\delta> \|v_0 \|_{C(\overline{\Omega})} e^{-1} $. Then we have
		\begin{align} \label{l3m}
			u_k \geq e^{-\delta t} \underline{u} \quad \mathrm{in} \ \overline{Q_{T}},
		\end{align}
	where $\underline{u}$ satisfies \eqref{heatunder}-\eqref{l2b}.
	\end{lemm}
	\begin{proof}
		We first consider the case of $2<m<3$.
		Set $W:= u_k- e^{-\delta t} \underline{u}+\e$ for any constant $\e \in (0,\e_k)$, where $\e_k$ will be chosen later. By \eqref{heatunder} and \eqref{main2a},
		the function $W$ satisfies
		\begin{align} \label{l3e1}
			\begin{cases}
			\partial_t W = \Delta W -kv_0u_k^m e^{-k\int_{0}^{t} u_k \ d\tau }+\delta e^{-\delta t} \underline{u}& \  \mathrm{in} \ Q_T,\\
			\partial_\nu W=0 & \  \mathrm{on} \ S_T,\\
			W(\cdot,0)\geq \e  & \  \mathrm{on} \ \Omega.
			\end{cases}
		\end{align}
	 We prove that $W$ is positive in $Q_{T}$ by a contradiction argument.
	Suppose that the point $(x_0,t_0) \in Q_{T}$ satisfies
	\begin{align}\label{W}
		\begin{cases}
			W(x_0,t_0)=0,\\
			W(x,t)>0 \quad \mathrm{for}\ (x,t) \in \Omega \times[0,t_0).
		\end{cases}
	\end{align}
Since $(x_0,t_0)$ is a minimum point of $W$, we have $\partial_t W (x_0,t_0) \leq 0$ and $\Delta W (x_0,t_0)\geq 0$. If $x_0 \in \Omega \backslash \supp(v_0)$, we have
\begin{align*}
	-  kv_0u_k^m e^{-k\int_{0}^{t_0} u_k \ d\tau } + \delta e^{-\delta t_0} \underline{u}=\delta    e^{-\delta t_0} \underline{u} >0
\end{align*}
by \eqref{l2a}. Since
\begin{align}\label{l3e1b}
	\begin{split}
			0\geq& W_t(x_0,t_0) - \Delta W \\
			=&-  kv_0u_k^m e^{-k\int_{0}^{t_0} u_k \ d\tau } + \delta e^{-\delta t_0} \underline{u}\\
			>&0,
	\end{split}
\end{align} this contradicts. Hence we consider the case where $x_0 \in \supp (v_0)$. Set
\begin{align*}
	&I_1 := -kv_0(x_0) (u_k(x_0,t_0)^m- e^{-m\delta t_0}\underline{u}(x_0,t_0)^m )
	e^{-k\int_{0}^{t_0} u_k(x_0,\tau) \ d\tau },\\
	&I_2 := kv_0(x_0) e^{-m\delta t_0} \underline{u}(x_0,t_0)^m
	\left( e^{-k\int_{0}^{t_0} e^{-\delta \tau} \underline{u} (x_0,\tau) \ d\tau }-e^{-k\int_{0}^{t_0} u_k (x_0,\tau) \ d\tau } \right),\\
	&I_3:= e^{-\delta t_0}\underline{u}(x_0,t_0)
	(\delta - kv_0(x_0) e^{-(m-1)\delta t_0}\underline{u}(x_0,t_0)^{m-1}
	e^{-k\int_{0}^{t_0} e^{-\delta \tau} \underline{u} (x_0,\tau) \ d\tau }) .
\end{align*}Then we have 
\begin{align}\label{l3e14b}
	-  kv_0u_k(x_0,t_0)^m e^{-k\int_{0}^{t_0} u_k(x_0,\tau) \ d\tau } + \delta e^{-\delta t_0}  \underline{u}(x_0,t_0) =I_1+I_2+I_3.
\end{align}
Since $-\e = u_k (x_0,t_0)-e^{-\delta t_0}\underline{u}(x_0,t_0)$, we have
\begin{align}\label{l3e13}
	I_1 \geq k v_0 e^{-k\int_{0}^{t_0} u_k \ d\tau } m u_k^{m-1} \e \geq 0.
\end{align}
Since $ e^{-\delta t}\underline{u}-u_k \leq \e$ on $\overline{\Omega} \times [0,t_0]$,
we have
\begin{align} \label{l3e14}
	\begin{split}
		I_2 &\geq kv_0 e^{-m\delta t_0} \underline{u}^m e^{-k\int_{0}^{t_0}e^{-\delta \tau}  \underline{u} \ d\tau } (1-e^{k\e t_0}) 
		\\&\geq -kv_0 e^{-m\delta t_0} \underline{u}^m(e^{k\e t_0}-1).
	\end{split}
\end{align}
Set $z_k := k e^{-(m-1)\delta t_0}\underline{u}(x_0,t_0)^{m-1}$. Then it follows that
\begin{align*}
	I_3=e^{-\delta t_0}\underline{u}(x_0,t_0)
	(\delta - v_0(x_0) z_k e^{-z_k}
	e^{z_k-k\int_{0}^{t_0} e^{-\delta \tau} \underline{u} (x_0,\tau) \ d\tau }).
\end{align*}
Since $\underline{u}(x_0,0)=0$ in $\supp (v_0)$, we compute
\begin{align} \label{l3e3}
	\begin{split}
		&k^{-1}\left( z_k-k\int_{0}^{t_0} e^{-\delta \tau} \underline{u} (x_0,\tau) \ d\tau  \right)\\
		=& \int_{0}^{t_0} (m-1)e^{-(m-1)\delta \tau}
		\underline{u} (x_0,\tau)^{m-2}\partial_t \underline{u} (x_0,\tau)  
		\\ 
		& -(m-1)\delta e^{-(m-1)\delta \tau}\underline{u} (x_0,\tau)^{m-1} -e^{-\delta \tau} \underline{u} (x_0,\tau) \ d\tau  \\
		\leq& \int_{0}^{t_0} e^{-(m-1)\delta \tau}\underline{u} (x_0,\tau)^{m-2} \left(
		(m-1)
		\partial_t \underline{u} (x_0,\tau)  - \underline{u} (x_0,\tau)^{3-m} 
		\right) d\tau 
	\end{split}
\end{align}
By \eqref{l2b} and \eqref{l3e3}, we have
\begin{align}\label{l3e16}
	e^{z_k-k\int_{0}^{t_0} e^{-\delta \tau} \underline{u} (x_0,\tau) \ d\tau } \leq 1.
\end{align}
Since the maximum value of the function $se^{-s}$ on $[0,\infty)$ is $e^{-1}$, by \eqref{l3e16}, we have
\begin{align}\label{l3e17}
	I_3 \geq e^{-\delta t_0}\underline{u}(x_0,t_0)
	(\delta - v_0(x_0)  e^{-1}
	)
\end{align}
By \eqref{l3e13}, \eqref{l3e14}, and\eqref{l3e17}, we obtain
\begin{align*}
	\begin{split}
		I_1+I_2+I_3 &\geq e^{-\delta t_0}\underline{u}
		(\delta - \frac{v_0}{e} -kv_0  e
		^{-(m-1)\delta t_0} \underline{u}^{m-1}(e^{k\e t_0}-1)
		).
	\end{split}
\end{align*}
We now specify $\e_k>0$. Since $\delta >\|v_0 \|_{C(\overline{\Omega})} e^{-1}$, we suppose that $\e_k$ satisfies
\begin{align*}
	\delta - \frac{\|v_0 \|_{C(\overline{\Omega})}}{e} - k\|v_0 \|_{C(\overline{\Omega})} \|u_0 \|_{C(\overline{\Omega})}^{m-1}(e^{k \e_k T}-1)>0. 
\end{align*}
Then it follows that
\begin{align} \label{l3e20}
	I_1+I_2+I_3>0
\end{align}
for any $\e \in (0,\e_k)$. This contradicts by \eqref{l3e1b},
\eqref{l3e14b}, and \eqref{l3e20}. Hence the function $W =u_k- e^{-\delta t} \underline{u}+\e$ is positive on $Q_{T}$
for any $0<\e<\e_k$. If $\e$ close to $0$, we obtain $u_k\geq e^{-\delta t} \underline{u}$
on $Q_{T}$. Since $u_k$ and $e^{-\delta t} \underline{u}$ are continuous functions, this inequality holds on $\overline{Q_{T}}$.
We next consider the case of $m=3$.
Then we similarly obtain \eqref{l3e13} and \eqref{l3e14}.
Instead of \eqref{l3e3}, we obtain
\begin{align*} 
	\begin{split}
		&k^{-1}\left( z_k-k\int_{0}^{t_0} e^{-\delta \tau} \underline{u} (x_0,\tau) \ d\tau  \right)\\
		\leq& \int_{0}^{t_0} e^{-2\delta \tau}\underline{u} (x_0,\tau) \left(
		2
		\partial_t \underline{u} (x_0,\tau)  - 1
		\right) d\tau \\
		\leq&0
	\end{split}
\end{align*}
By \eqref{l2b}. Then we have \eqref{l3e16}.
Hence we can similarly prove that \eqref{l3m} holds.
Thus this lemma follows.
	\end{proof}
	\begin{lemm}\label{l4l}
		Suppose that $u_\infty \in C^{2,1}(\overline{Q_{T}})$
		satisfies \eqref{heat}. Then we have
		\begin{align}\label{l4a}
			u_\infty (x,t)-\| u_\infty \|_{C^{2,1}(\overline{Q_{T}})} t
			\leq u_k(x,t)\leq u_\infty (x,t).
		\end{align}
	\end{lemm}
	\begin{proof}
		Using the comparison principle, we obtain
		\begin{align}\label{l4e1}
			u_k(x,t)\leq u_\infty (x,t).
		\end{align}
	Set $W(x,t):=u_k(x,t)-u_\infty (x,t)+\| u_\infty \|_{C^{2,1}(\overline{Q_{T}})}t+\e$ for any constant $\e>0$. 
	By \eqref{heat} and \eqref{main2a},
	the function $W$ satisfies
	\begin{align} \label{l4e2}
		\begin{cases}
			\partial_t W = \Delta W -kv_0u_k^m e^{-k\int_{0}^{t} u_k \ d\tau }+\| u_\infty \|_{C^{2,1}(\overline{Q_{T}})}& \  \mathrm{in} \ Q_T,\\
			\partial_\nu W=0 & \  \mathrm{on} \ S_T,\\
			W(\cdot,0)= \e  & \  \mathrm{on} \ \Omega.
		\end{cases}
	\end{align}
We prove that $W$ is positive in $Q_{T}$ by a contradiction argument.
Suppose that the point $(x_0,t_0) \in Q_{T}$ satisfies \eqref{W}. Then we have $\partial_t W(x_0,t_0)\leq 0$
and $\Delta W (x_0,t_0 )\geq 0$. 
If $x_0 \in \Omega \backslash \supp(v_0)$, we have
\begin{align*}
	-  kv_0u_k^m e^{-k\int_{0}^{t_0} u_k \ d\tau } + \| u_\infty \|_{C^{2,1}(\overline{Q_{T}})}=\| u_\infty \|_{C^{2,1}(\overline{Q_{T}})}>0
\end{align*}
Since $u_0 \not\equiv 0$. It holds that
\begin{align*}
		0\geq W_t(x_0,t_0) - \Delta W 
		= \| u_\infty \|_{C^{2,1}(\overline{Q_{T}})}
		>0,
\end{align*} but this contradicts. If $x_0 \in \supp(v_0)$,
we have $u_0(x_0)=0$ since $u_0v_0 \equiv0$. Then we obtain
\begin{align*}
	\begin{split}
		0=W(x_0,t_0)&=u_k(x_0,t_0)-u_\infty (x_0,t_0)+\| u_\infty \|_{C^{2,1}(\overline{Q_{T}})}t_0+\e \\
		&\geq u_k(x_0,t_0) -\| u_\infty \|_{C^{2,1}(\overline{Q_{T}})}t_0+\| u_\infty \|_{C^{2,1}(\overline{Q_{T}})}t_0+\e\\
		&\geq \e>0,
	\end{split}
\end{align*}
but this contradicts. Hence the function $W =u_k-u_\infty (x,t)+\| u_\infty \|_{C^{2,1}(\overline{Q_{T}})}t+\e$ is positive on $Q_{T}$
for any $\e>0$. Since $u_k$ and $u_\infty$ are continuous functions, the inequality $u_\infty (x,t)-\| u_\infty \|_{C^{2,1}(\overline{Q_{T}})} t
\leq u_k(x,t)$ holds on $\overline{Q_{T}}$.
Thus this lemma follows.
	\end{proof}
	\begin{lemm}\label{l5l}
		There exist a constant $0<k^*=k^*(u_0,\|v_0 \|_{C(\overline{\Omega})},T)$ and a sequence $\{t_k\}_{k \in \Na}$ such that, 
\begin{align} \label{l5a}
	\begin{split}
		&\lim_{k \to \infty} t_k =0,\\
	&	kv_0u_k^m e^{-k\int_{0}^{t} u_k \ d\tau } \leq \frac{1}{k^{\frac{1}{2}}} \quad \mathrm{in} \ \Omega \times [t_k,T],
	\end{split}
\end{align}
for any $k>k^*$.
	\end{lemm}
	\begin{proof}
		If $x\in \Omega \backslash \supp(v_0)$, we have
		\begin{align*}
			kv_0u_k^m e^{-k\int_{0}^{t} u_k \ d\tau }=0\leq
			\frac{1}{k^{\frac{1}{2}}}.
		\end{align*}
	Hence we consider the case where $x\in \supp(v_0)$.
	Using the inequality $s^2 e^{-s}\leq \frac{4}{e^2}$ for $s\geq 0$, we obtain
	\begin{align}\label{l5e24}
		\begin{split}
				&kv_0(x)u_k(x,t)^m e^{-k\int_{0}^{t} u_k(x,\tau) \ d\tau }\\
		\leq& 	kv_0(x)u_k(x,t)^m \frac{4}{\left(ke\int_{0}^{t} u_k(x,\tau) \ d\tau  \right)^2}.
		\end{split}
	\end{align}
	Set $\delta =2 \|v_0 \|_{C(\overline{\Omega})} e^{-1}$. Define
	\begin{align}\label{defgamma}
		\gamma(t):= \int^t_0 \min_{x\in \overline{\Omega}}
		e^{-\delta \tau}\underline{u} \ d\tau,
	\end{align}
	where $\underline{u}$ satisfies \eqref{heatunder}-\eqref{l2b}. By \eqref{l3m} and \eqref{l5e24}, it follows that
	\begin{align}\label{l5e25}
		kv_0(x)u_k(x,t)^m e^{-k\int_{0}^{t} u_k(x,\tau) \ d\tau } \leq \frac{4\|v_0 \|_{C(\overline{\Omega})} \|u_0 \|_{C(\overline{\Omega})}^{m}}{ke^2 \gamma(t)^2}.
	\end{align}
By \eqref{l2a}, the function $\gamma(t)$ is a strictly increasing continuous function in $(0,T]$.
Set
\begin{align*}
	\tilde{k}:= \gamma(T)^{-8}.
\end{align*}
The function $\tilde{T}:(\tilde{k},\infty)\to \R$ is defined by
\begin{align*}
	\tilde{T}(s):= \gamma^{-1}(s^{-\frac{1}{8}}).
\end{align*}
Set $t_k:= \tilde{T}(k) $. Then we have
\begin{align}\label{l5e26}
	\gamma(t_k)=\gamma(\gamma^{-1}(k^{-\frac{1}{8}}))=k^{-\frac{1}{8}}.
\end{align}
Since $\gamma(t)$ is a strictly function, we obtain
\begin{align}\label{l5e27}
	0<\gamma(t_k) \leq \gamma(t)
\end{align}
for any $t\in [t_k,T]$.
Set
\begin{align*}
	k^*:= \max \left\{\tilde{k}, \left(\frac{4\|v_0 \|_{C(\overline{\Omega})} \|u_0 \|_{C(\overline{\Omega})}^{m}}{e^2 }\right)^4 \right\}
\end{align*}
 By \eqref{l5e25}-\eqref{l5e27}, for any $k>k^*$, we obtain
\begin{align*}
	kv_0(x)u_k(x,t)^m e^{-k\int_{0}^{t} u_k(x,\tau) \ d\tau } \leq \frac{4\|v_0 \|_{C(\overline{\Omega})} \|u_0 \|_{C(\overline{\Omega})}^{m}}{k^{\frac{3}{4}}e^2 }
	\leq \frac{1}{k^{\frac{1}{2}}}
\end{align*}
in $\Omega\times[t_k,T]$. By definition of $t_k$, it follows that  $\lim_{k \to \infty} t_k=0.$ Thus this lemma follows.
	\end{proof}

\begin{proof}[Proof of Theorem \ref{theorem1}]
By \eqref{l4a}, we have
	\begin{align*}
		u_k \leq u_\infty \quad \mathrm{in} \ \overline{Q_T}.
	\end{align*}
	We next construct $\{U_k\}_{k\in \Na}$ satisfying
	$U_k \to u_\infty$ in $C(\overline{Q_T})$ and $U_k \leq u_k$ for sufficiently large $k>0$. By \eqref{l4a} again, it follows that
	\begin{align}\label{eq28}
			u_\infty (x,t)-\| u_\infty \|_{C^{2,1}(\overline{Q_{T}})} t
		\leq u_k(x,t) \quad \mathrm{in} \ \overline{Q_T}.
	\end{align}
Let $\{t_k\}_{k\in \Na}$ be a sequence satisfying \eqref{l5a}.
Define
\begin{align*}
	\underline{U}_k(x,t):=u_\infty(x,t)-\| u_\infty \|_{C^{2,1}(\overline{Q_{T}})} t_k-k^{-\frac{1}{2}}(t-t_k)
\end{align*}
for $(x,t)\in \Omega \times (t_k,T)$. Then it follows that
\begin{align*} 
	\begin{cases}
		\partial_t \underline{U}_k=	\Delta \underline{U}_k- k^{-\frac{1}{2}} & \mathrm{in}  \  \Omega \times (t_k,T),\\
		\partial_\nu \underline{U}_k =0 & \mathrm{on}  \  S_T\backslash S_{t_k},\\
		\underline{U}_k(x,t_k)=u_\infty (x,t_k)-\| u_\infty \|_{C^{2,1}(\overline{Q_{T}})} t_k   & \mathrm{in}  \  \Omega.
	\end{cases}
\end{align*}
Using the comparison principle and Lemma \ref{l5l}, we have
\begin{align}\label{eq30a}
	u_k(x,t) \geq \underline{U}_k \quad \mathrm{in} \ \Omega \times [t_k,T]
\end{align}
for any $k>k^*$. Define
\begin{align*}
	U_k(x,t):=\begin{cases}
		u_\infty (x,t)-\| u_\infty \|_{C^{2,1}(\overline{Q_{T}})} t & (0\leq t\leq t_k),\\
		\underline{U}_k (x,t) &(t_k< t\leq T).
	\end{cases}
\end{align*}
By \eqref{eq28} and \eqref{eq30a}, we obtain
\begin{align*}
	U_k \leq u_k \quad \mathrm{in} \  \overline{Q_T}
\end{align*}
for any $k>k^*$. We next show that $U_k \to u_\infty$ in $C(\overline{Q_T})$.
By definition of $U_k$, we obtain
\begin{align}
	\begin{split}
	&	\sup_{(x,t)\in \overline{Q_T}} |u_\infty (x,t)-U_k(x,t)|\\
	\leq&\sup_{(x,t)\in \overline{Q_{t_k}}} |u_\infty (x,t)-U_k(x,t)|+
	\sup_{(x,t)\in \overline{\Omega}\times[t_k,T]} |u_\infty (x,t)-U_k(x,t)|\\
	\leq& \| u_\infty \|_{C^{2,1}(\overline{Q_{T}})} t_k+
	\| u_\infty \|_{C^{2,1}(\overline{Q_{T}})} t_k+k^{-\frac{1}{2}}T\\
	\to& 0 \quad (k\to \infty)
	\end{split}
\end{align}
by \eqref{l5a}.
Hence we obtain
\begin{align*}
	u_k \to u_\infty \quad \mathrm{in} \ C(\overline{Q_T})
\end{align*}
as $k$ tends to infinity.
We next show that $v_k$ converges uniformly to $0$ in $\overline{\Omega} \times[\rho,T]$ for any $\rho \in (0,T)$. There exists a constant $k_1>k^*$ satisfying
$t_{k_1}<\rho$. Set $\delta=2\|v_0 \|_{C(\overline{\Omega})}e ^{-1}$. Let $\underline{u}$ be a function satisfying \eqref{heatunder}-\eqref{l2b}. Using Lemma \ref{l3l}, for any $(x,t)\in \overline{\Omega} \times[\rho,T]$, we have
\begin{align*}
	0 & \leq \lim_{k \to \infty} v_k \\
	&= \lim_{k \to \infty}v_0\exp \left(-k \int^{t}_0 u_k \ d\tau \right)\\
	&\leq \lim_{k \to \infty}  v_0\exp \left(-ke^{-\delta T} \int^{\rho}_0 \min_{x\in \overline{\Omega}} \underline{u} \ d\tau \right)\\
	&\leq 0.
\end{align*}
Hence  $v_k$ converges uniformly to $0$ in $\overline{\Omega} \times[\rho,T]$. Thus this theorem follows.
\end{proof}
\section{Dirichlet boundary problem}
By modifying the theorem for the Neumann boundary problem, we show that we can apply it to the Dirichlet boundary problem.
For Lemma \ref{l1l} and Lemma \ref{l2l}, if $m>3$, we obtain the same inequality as when $m = 3$. By Dirichlet boundary condition, instead of \eqref{l2a}, it follows that
\begin{align*}
	\underline{u}(x,t)>0 \quad\mathrm{in} \ \Omega \times (0,T].
\end{align*}
With this modification, Lemma \ref{l3l} and Lemma \ref{l4l} also hold under the Dirichlet boundary condition.
We consider Lemma \ref{l5l}. The function $\gamma(t)$ of \eqref{defgamma} is zero under the Dirichlet boundary condition. Thus instead of Lemma \ref{l5l}, we use the following lemma.
\begin{lemm}\label{l5aaa}
		Suppose that $\Omega_0$ satisfies $\Omega_0\subset \subset \Omega$. There exist a constant $0<k^*=k^*(u_0,\|v_0 \|_{C(\overline{\Omega})},T,\Omega_0, \Omega)$ and a sequence $\{t_k\}_{k \in \Na}$ such that, 
	\begin{align*} 
		\begin{split}
			&\lim_{k \to \infty} t_k =0,\\
			&	kv_0u_k^m e^{-k\int_{0}^{t} u_k \ d\tau } \leq \frac{1}{k^{\frac{1}{2}}} \quad \mathrm{in} \ \Omega_0 \times [t_k,T],
		\end{split}
	\end{align*}
	for any $k>k^*$.
\end{lemm}
\begin{proof}
	Define
	\begin{align}\label{defgamma2}
		\gamma(t):= \int^t_0 \min_{x\in \overline{\Omega_0}}
		e^{-\delta \tau}\underline{u} \ d\tau,
	\end{align}
Using the strong maximum principle, we have 
\begin{align*}
	\min_{x\in \overline{\Omega_0}} \underline{u}(x,t)>0.
\end{align*}
Hence the function $\gamma(t)$ is a strictly increasing continuous function in $(0,T]$. Thus this lemma follows by applying \eqref{defgamma2} instead of \eqref{defgamma}.
\end{proof}
\begin{proof}[Proof of Theorem \ref{theorem2}]
	Let $\e>0$. Define
	\begin{align*}
		\Omega_\e := \{x\in \Omega : \dist(x, \partial \Omega)>\e \}.
	\end{align*}
For sufficiently small $\e>0$, we have $\Omega_0 \subset \subset \Omega_\e \subset \subset \Omega$.
Since $u_\infty=0$ on $S_T$ and $0\leq u_k\leq u_\infty$ in $\overline{Q_T}$, we obtain
\begin{align}\label{t2eq1}
	\sup_{(x,t)\in \Omega \backslash \Omega_\e \times[0,T]}
	|u_\infty(x,t)-u_k(x,t)|\leq \|u_\infty \|_{C^{2,1}(\overline{Q_T})}\e.
\end{align}
Define
\begin{align*}
	\underline{U}_k(x,t):=u_\infty(x,t)-\| u_\infty \|_{C^{2,1}(\overline{Q_{T}})} t_k-k^{-\frac{1}{2}}(t-t_k)
	-\|u_\infty \|_{C^{2,1}(\overline{Q_T})}\e
\end{align*}
for $(x,t)\in \Omega \times (t_k,T)$. Then we have $\underline{U}_k\leq 0$ on $\partial \Omega_\e$.
Define
\begin{align*}
	U_k(x,t):=\begin{cases}
		u_\infty -\| u_\infty \|_{C^{2,1}(\overline{Q_{T}})} t-\|u_\infty \|_{C^{2,1}(\overline{Q_T})}\e & (0\leq t\leq t_k),\\
		\underline{U}_k(x,t) &(t_k< t\leq T).
	\end{cases}
\end{align*}
Using the comparison principle and Lemma \ref{l5aaa}, we have
\begin{align}\label{t2eq3}
	U_k(x,t) \leq u_k(x,t) \quad \mathrm{in} \ \overline{\Omega_\e} \times [0,T]
\end{align}
for any $k>k^*$.
By \eqref{t2eq1}, \eqref{t2eq3}, and definition of $U_k$, we obtain
\begin{align*}
	\begin{split}
		&	\sup_{(x,t)\in \overline{Q_T}} |u_\infty (x,t)-u_k(x,t)|\\
		\leq&\sup_{(x,t)\in \Omega \backslash \Omega_\e \times[0,T]} |u_\infty (x,t)-u_k(x,t)|+
		\sup_{(x,t)\in \overline{\Omega_\e}\times[0,t_k]} |u_\infty (x,t)-U_k(x,t)|\\
		&+
		\sup_{(x,t)\in \overline{\Omega_\e}\times[t_k,T]} |u_\infty (x,t)-U_k(x,t)|\\
		\leq& \| u_\infty \|_{C^{2,1}(\overline{Q_{T}})} \e +
		\| u_\infty \|_{C^{2,1}(\overline{Q_{T}})} t_k+
		\| u_\infty \|_{C^{2,1}(\overline{Q_{T}})} \e \\
		&+\| u_\infty \|_{C^{2,1}(\overline{Q_{T}})} t_k
		+k^{-\frac{1}{2}}T+\| u_\infty \|_{C^{2,1}(\overline{Q_{T}})} \e.
	\end{split}
\end{align*}
Since $t_k \to 0$, it follows that
\begin{align*}
	\sup_{(x,t)\in \overline{Q_T}} |u_\infty (x,t)-u_k(x,t)|
	\leq (1+3\| u_\infty \|_{C^{2,1}(\overline{Q_{T}})})\e
\end{align*}
for sufficiently large $k>k^*$.
We next show that $v_k$ converges uniformly to $0$ in $\overline{\Omega_0} \times[\rho,T]$. Set $\delta=2\|v_0 \|_{C(\overline{\Omega})}e ^{-1}$. For any $(x,t)\in \overline{\Omega_\e} \times[\rho,T]$, we have
\begin{align*}
	0 & \leq \lim_{k \to \infty} v_k \\
	&= \lim_{k \to \infty}v_0\exp \left(-k \int^{t}_0 u_k \ d\tau \right)\\
	&\leq \lim_{k \to \infty}  v_0\exp \left(-ke^{-\delta T} \int^{\rho}_0 \min_{x\in \overline{\Omega_\e}} \underline{u} \ d\tau \right)\\
	&\leq 0.
\end{align*}
Hence  $v_k$ converges uniformly to $0$ in $\overline{\Omega_\e} \times[\rho,T]$. Thus this theorem follows.

\end{proof}

\end{document}